\documentclass[12pt,a4paper,twoside,reqno]{amsart}
\usepackage{fancyhdr}
\usepackage{amssymb,mathrsfs,amsmath,amsthm}
\usepackage[numbers,sort&compress]{natbib}
\usepackage{authblk}
\usepackage{url}
\usepackage[numbers]{natbib}
\usepackage{setspace}
\usepackage{mathtools}
\usepackage{comment}
\usepackage[inline]{enumitem}
\usepackage[
  colorlinks=true,
  linkcolor=blue,
  citecolor=blue,
  urlcolor=blue,
  pagebackref,
  pdfauthor={Your Name},
  pdftitle={Your Document Title},
  pdfkeywords={keyword1, keyword2, keyword3}
]{hyperref}

\usepackage[margin=1in]{geometry} 

\newtheorem{thm}{Theorem}[section]
\newtheorem{cor}[thm]{Corollary}
\newtheorem{lem}[thm]{Lemma}
\newtheorem{prop}[thm]{Proposition}
\theoremstyle{definition}

\newtheoremstyle{boldremark}
  {10pt}  
  {10pt}   
  {}      
  {}       
  {\bfseries} 
  {.}      
  { }     
  {}  
\theoremstyle{boldremark}
\newtheorem{Remark}[thm]{\textbf{Remark}}
\numberwithin{equation}{section}
\newenvironment{mathclass}
  {Mathematics Subject Classification (2020):}
 
\newcommand{\R}{\mathbb{R}}
\newcommand{\C}{\mathbb{C}}

\renewcommand{\keywords}[1]{%
  \par\noindent
  Keywords: #1
  \par
}

\def\R {\mathbb{R}}

\newcommand{\be}{\begin{equation}}
\newcommand{\ee}{\end{equation}}
\newcommand{\bea}{\begin{eqnarray}}
\newcommand{\eea}{\end{eqnarray}}
\newcommand{\Bea}{\begin{eqnarray*}}
\newcommand{\Eea}{\end{eqnarray*}}
\newcommand{\bt}{\begin{Theorem}}
\newcommand{\et}{\end{Theorem}}

\newcommand{\bpr}{\begin{Proposition}}
\newcommand{\epr}{\end{Proposition}}

\newcommand{\bl}{\begin{Lemma}}
\newcommand{\el}{\end{Lemma}}
\newcommand{\bi}{\begin{itemize}}
\newcommand{\ei}{\end{itemize}}

\newtheorem{Definition}{Definition}[section]
\newtheorem{Theorem}[Definition]{Theorem}
\newtheorem{Lemma}[Definition]{Lemma}
\newtheorem{Proposition}[Definition]{Proposition}

\title
{Low-Regularity Global solution of the inhomogeneous nonlinear Schr\"odinger equations\\ in modulation spaces  }
\author{Divyang G. Bhimani}
\address{Divyang G. Bhimani\\Department of Mathematics\\Indian Institute of Science Education and Research\\ Pune 411008\\India}
\email{divyang.bhimani@iiserpune.ac.in}

\author{Diksha Dhingra}
\address{Diksha Dhingra \\Department of Mathematics\\ Indian Institute of Technology\\ Indore, 452020\\ India}
\email{dikshadd1996@gmail.com}

\author{Vijay Kumar Sohani}
\address{Vijay Kumar Sohani\\ Department of Mathematics\\ Indian Institute of Technology\\ Indore 452020\\ India}
\email{vsohani@iiti.ac.in}
\fancypagestyle{plain}{\fancyhf{} }
\begin{document}
\date{}
\maketitle{}
\begin{center}
DIVYANG G. BHIMANI, DIKSHA DHINGRA\footnote{Corresponding author: Diksha Dhingra.}, VIJAY KUMAR SOHANI
\end{center}
\begin{abstract}
The study of low regularity Cauchy data for nonlinear dispersive PDEs has successfully been achieved using modulation spaces $M^{p,q}$ in recent years. In this paper, we study the inhomogeneous nonlinear Schr\"odinger equation (INLS) 
\[iu_t + \Delta u\pm |x|^{-b}|u|^{\alpha}u=0\ \ (b, \alpha> 0) \]
 on whole space $\mathbb R^n$ in modulation spaces. In the  subcritical regime $(0<\alpha< \frac{4-2b}{n}),$ we establish local well-posedness  in $L^{2}+M^{\alpha+2,\frac{\alpha+2}{\alpha+1}}( \supset L^2 + H^s \ \text{for} \ s>\frac{n\alpha}{2(\alpha+2)}).$ By adapting Bourgain’s high-low decomposition method, we establish global well-posedness in $M^{p,\frac{p}{p-1}}$ with $2<p$ and $p$ sufficiently close to 2. This is the first global well-posedness result for INLS on modulation spaces, which  contains certain Sobolev  $H^s$  $(0<s<1)$ and  $L^p_s-$Sobolev spaces.
\end{abstract}
\footnotetext{\begin{mathclass}
Primary 35Q55, 35Q60, 42B37; Secondary 35A01.
\end{mathclass}
}

\section{Introduction}
\subsection{Background}\label{bcg} We study the inhomogeneous nonlinear Schr\"odinger equation (INLS for short) of the following form
\begin{equation}\label{INLS}
\begin{cases}
    iu_t(x,t) + \Delta u(x,t)+\mu |x|^{-b}(|u|^{\alpha}u)(x,t)=0\\ u(x,0)=u_0(x)
\end{cases}(x , t ) \in \mathbb R^n \times \mathbb R, 
\end{equation}
where $u(x,t) \in \mathbb C,  \mu = \pm 1,\; \alpha,  \; b>0 $ and $\Delta $ is a Laplacian operator. In this paper, we assume that $0<b< \min \{2, n \}$ unless it is explicitly specified. The parameters $\mu = 1$ (resp. $\mu = -1$) corresponds to the focusing (resp. defocusing) case.

The classical nonlinear Schr\"odinger equation (NLS for short, $b=0$ case)  has been extensively studied over the past three decades.  See \cite{TaoBook, linares2020, cazenave2003semilinear, wang2011harmonic}. This appeared  in nonlinear optics \cite{gill2000optical, liu1994laser} and  Bose-Einstein condensate  (BEC) \cite{condensed_matter_paper}. 
The NLS governing beam propagation does not support stable high power propagation in a homogeneous bulk medium. In \cite{gill2000optical,liu1994laser,rotating}, it was proposed that stable high power propagation could be achieved in plasma by sending a preliminary laser beam to create a channel with reduced electron density. This channel mitigates the nonlinear effects within it, thereby allowing for stable propagation of the high-power beam. Considering these conditions, the beam propagation can be modeled by the inhomogeneous NLS of the form
\begin{equation}\label{INLSk}
 iu_t + \Delta u+K(x)|u|^{\alpha}u=0.
\end{equation}
Here $u$ is the electric field in optics, $\alpha  > 0 $ is the power of nonlinear interaction, and the potential $K(x)$ is proportional to the electron density. INLS \eqref{INLS} plays an important role as a limiting equation in the analysis of \eqref{INLSk} with $K(x) \sim  |x|^{-b}$ as $|x|  \rightarrow  \infty$; see \cite{genoud2008schrodinger}. INLS \eqref{INLS} has also sparked significant theoretical interest on nonlinear phenomena in BEC with spatially inhomogeneous interactions; see \cite{spatially1}.
 Due to its vast applications, in recent years, the theory of well-posedness,  scattering,  blow-up  analysis, etc.  for INLS \eqref{INLS} has been  studied extensively. See e.g. \cite{guzman2020,genoud2008schrodinger,genoud2010bifurcation,genoud2012inhomogeneous,an2021local,an2021small, JMPAMS,Farahguzman,JMNA,GuzmanMurphy, JMSIAM, LuizMathZ, dinh2021long, Dinhradial, Merle, CarlesINLS}. 
\\
\par{Formally, the solution to  INLS \eqref{INLS} conserves both mass and energy:
\begin{align}
M[u(t)] &= \int_{\mathbb{R}^n} |u(x, t)|^2 \, dx = M[u_0], \label{mass}\\
 E[u(t)] &= \int_{\mathbb{R}^n} |\nabla u( x,t)|^2 dx- \frac{\mu}{\alpha + 2}  \left|\left||x|^{-b}|u|^{\alpha+2}\right|\right|_{L^{1}_{x}}= E[u_0].\nonumber
\end{align}}
 If $u=u(x,t)$ is a solution of INLS \eqref{INLS}, then so is $u_{\lambda}(x,t)= \lambda^{\frac{2-b}{\alpha}} u (  \lambda x,\lambda^2 t)$ (scaling)
 also a solution of INLS \eqref{INLS} with initial data $u_{\lambda}(x,0)=\lambda^{\frac{2-b}{\alpha}} u (  \lambda x,0).$ Computing homogeneous Sobolev norm we  have
\[ \|u_{\lambda}(\cdot,0)\|_{\dot{H}^s}= \lambda^{s- \frac{n}{2}+ \frac{2-b}{ \alpha}} \|u_0\|_{\dot{H}^s}.\]
 The $\dot{H}^{s}-$norm is invariant under the above scaling  for $s=s_{b}=\frac{n}{2}- \frac{2-b}{ \alpha}$ (called critical Sobolev index).  We say  INLS \eqref{INLS} is 
\begin{equation*}
   L^2  \ (\dot{H}^1)  -\begin{cases}
        \text{subcritical} & if  \  s_b<0 \ (s_b=1)\\
        \text{critical} & if \ s_b=0 \ (s_b=1)\\
        \text{supercritical} & if \  s_b>0 \ (s_b>1).
    \end{cases}
\end{equation*}

We have mass critical or $L^{2}$-critical case if $s_{b}=0 \;(\alpha=\frac{4-2b}{n}).$ If $ s_{b}=1\;(\alpha=\frac{4-2b}{n-2}),$ we have energy critical or $\dot{H}^{1}-$ critical case. Finally, the problem is known to be mass-supercritical and energy-subcritical or intercritical if
$0 < s_{b} < 1.$ For $b=0$, we write $s_{0}=s_{c}.$  To state some known results, we set  notations:
$$\alpha_{s} =\begin{cases}
\frac{4 - 2b}{n - 2s} & \text{if } 0 \leq s < \frac{n}{2} \\
+\infty & \text{if } s \geq \frac{n}{2}
\end{cases}, \quad  \alpha_{s}^{*} =\begin{cases}
\frac{4 - 2b}{n - 2s} & \text{if } s < \frac{n}{2} \\
+\infty & \text{if } s =\frac{n}{2}
\end{cases},$$
$$\Tilde{2}=\begin{cases}
\frac{n}{3} & \text{if } n = 1, 2, 3 \\
2 & \text{if } n \geq 4
\end{cases} ~\text{and}~ \quad \hat{2}=\begin{cases} 
\min\left\{2, 1 + \frac{n - 2s}{2}\right\} & \text{if } n \geq 3 \\
n - s & \text{if } n = 1, 2.
\end{cases}$$
Now, we briefly summarize  some known  results of INLS \eqref{INLS}:
\begin{itemize}
    \item[-] Genoud and Stuart \cite[Theorem 1.1]{genoud2008schrodinger} proved:
\begin{itemize}
    \item locally well-posed (LWP for short) in $H^1$ for $0<\alpha < \alpha_1,  \mu=-1.$
    \item globally well-posed (GWP for short) for small data in $H^1$ when $ \alpha_0 \leq \alpha < \alpha_1, \mu=-1$.
    \item GWP for any data in $H^1$  for $0< \alpha< \alpha_0, \mu=-1.$
    \item GWP in $H^1$ for $\mu=1$ (see \cite{cazenave2003semilinear}).
\end{itemize}

\item[-] Guzm\'an \cite[Theorems 1.4, 1.8 and 1.9]{guzman2020} proved:
\begin{itemize}
    \item LWP in \( H^s \) for \( \max\{0,s_{b}\} < s \leq \min\left\{1, \frac{n}{2}\right\} \), \( 0 < b < \tilde{2} \), and \( 0 < \alpha < \alpha_{s}^{*} \).
    \item GWP for small data   in $H^s$ for \( s_{b} < s \leq \min\left\{1, \frac{n}{2}\right\} \), \( 0 < b < \tilde{2} \), \( \frac{4-2b}{n} < \alpha < \alpha_{s}^{*}\). 
    \item GWP in  $L^2$ for $0<\alpha< \frac{4-2b}{n}$. (For $b=0,$ see \cite[Theorem 1.1.]{YTsutsumi})
\end{itemize}
\item[-] JinMyoung An et al. \cite[Theorems 1.4 and 1.6]{an2021local}, \cite[Theorem 1.3]{an2021small}  expanded the range for \( s \) and \( 0 < b < \hat{2} \). Specifically,  INLS \eqref{INLS} is:
\begin{itemize}
    \item LWP in \( H^s \) \( \left( 0 \leq s < \min\{n, \frac{n}{2}+1\} \right) \) for \( 0 < \alpha < \alpha_{s} \). 
    \item GWP for small  data  in $H^s$  \( \left( 0 < s < \min\left\{n, \frac{n}{2} + 1\right\} \right) \) for  \( \alpha_{0} < \alpha < \alpha_{s} \).
\end{itemize}

\item[-]  Genoud \cite[Theorem 1] {genoud2012inhomogeneous} proved focusing mass-critical INLS \eqref{INLS} is GWP in \( H^1\) having data below ground state. This result was
extended to the intercritical case in \cite[Theorem 1.5]{farah2016global}.  See \cite{Dinhradial} and \cite{dinh2021long}.   We also refer the reader to \cite{JMSIAM} for the global behaviour of solutions to 3D focusing INLS \eqref{INLS} at the mass-energy threshold.

\item[-] Murphy \cite{JMPAMS} studied scattering below the ground state for the intercritical non-radial case in $H^1.$ See \cite{Farahguzman,JMNA,GuzmanMurphy} for more results on scattering.
On the other hand, for  blow-up analysis we refer (among others) to papers  by  Merle \cite{Merle} and Banica-Carles- Duyckaerts \cite{CarlesINLS}  and Cardoso-Farah \cite{LuizMathZ}.
\end{itemize}

Finally, taking these known results into account, in the next remark we would like to  highlight the following  points:
\begin{Remark}\label{sr}
\begin{enumerate}
    \item The  problem of GWP for large data in \( H^s \) with \( 0 < s < 1 \) remains unsolved.
    \item Most authors  studied INLS \eqref{INLS} in the $L^2-$based Sobolev spaces. It is natural to investigate similar theory  in $L_s^p-$Sobolev spaces.
\end{enumerate}
\end{Remark}
We shall see soon that modulation spaces will provide us an appropriate framework for the  concern raised in Remark \ref{sr}.
\subsection{Modulation spaces}
\label{mop} In the past two decades, \( M^{p,q}_s \)-spaces have been extensively explored as spaces for low regularity Cauchy data in the study of dispersive PDEs. See \cite{ Bhimani2016, Kasso2009,  KassoBook, BhimaniNorm, BhimaniHartree-Fock,bhimani2023mixed, BhiStroIll, Wang2006, Wang2007,RWZexe, LeonidIn, wang2011harmonic, leonidthesis}. \par{Let $\rho: \mathbb R^n \to [0,1]$  be  a smooth function satisfying   $\rho(\xi)= 1 \  \text{if} \ \ |\xi|_{\infty}\footnote{Define $|\xi|_{\infty}=\max\{ | \xi_i | : \xi= (\xi_1,..., \xi_n)\}.$}\leq \frac{1}{2} $ and $\rho(\xi)=
0 \  \text{if} \ \ |\xi|_{\infty}\geq  1$. Let  $\rho_k$ be a translation of $\rho,$ that is
$ \rho_k(\xi)= \rho(\xi -k) \ (k \in \mathbb Z^n).$
Denote 
$$\sigma_{k}(\xi)= \frac{\rho_{k}(\xi)}{\sum_{l\in\mathbb Z^{n}}\rho_{l}(\xi)}\quad (k \in \mathbb Z^n).$$
The frequency-uniform decomposition operators can be  defined by 
$$\square_k = \mathcal{F}^{-1} \sigma_k \mathcal{F} \quad (k \in \mathbb Z^n)$$ 
where $\mathcal{F}$ and $\mathcal{F}^{-1}$ denote the Fourier and inverse Fourier transform respectively.  
The weighted modulation spaces  $M^{p,q}_s \ (1 \leq p,q \leq \infty, s \in \R)$ is defined as follows:
  \begin{equation*}
    M^{p,q}_s= M^{p,q}_s(\R^n)= \left\{ f \in \mathcal{S}'(\R^n): \left|\left|f\right|\right|_{M^{p,q}_s}=  \left\| \|\square_kf\|_{L^p_x} (1+|k|)^{s} \right\|_{\ell^q_k}< \infty  \right\} . 
\end{equation*}
For $s=0,$ we write $M^{p,q}_0= M^{p,q}.$ See Remark \ref{edm}. For $p=q=2,$ modulation spaces coincide with Sobolev spaces, i.e. 
$M^{2,2}_s= H^s \ (s \in \R). $ For $p\in [1, \infty]$,  we denote $p'$ the H\"older conjugate, i.e.  $\frac{1}{p}+\frac{1}{p'}=1.$ 
By Lemma \ref{srp} \eqref{srp1} and \eqref{srp2}, we have 
\begin{equation}\label{use}
   H^s \subset  M^{\alpha+2, (\alpha+2)'}   \text{ for}  \ s > \frac{n\alpha}{2(\alpha+2)}. 
\end{equation}
Recall  $L^p_s-$Sobolev norm: $\|f\|_{L_s^p}= \| \mathcal{F}^{-1} \langle \cdot \rangle^s \mathcal{F}f\|_{L^p}$. By Proposition \ref{exa}, for $s > n \left( \frac{1}{p'}-\frac{1}{p}\right)$ and $ p \in [2,  \infty],$ we have
\begin{equation}\label{se}
  L^p_s\hookrightarrow  M^{p,p'} \hookrightarrow  L^p.  
\end{equation}
In fact, modulation spaces accommodate rougher Cauchy data compared to any fractional Bessel potential space. See e.g \cite[Chapter 6]{KassoBook}, \cite[Chapter 6]{wang2011harmonic}.
 The free Schr\"odinger propagator $e^{-it\Delta}: M^{p,q} \to M^{p,q}$ is a bounded operator for all $p, q$, see Proposition \ref{mpq}, while it is unbounded in $L^p_{s}-$ Sobolev spaces for $p\neq 2.$  Another great advantage is that $e^{-it\Delta}$ enjoys a truncated decay in modulation spaces, but its decay from $L^{p'}$ to $L^p$ contains singularity at $t=0$.   See    Proposition \ref{mpq}.   Using these facts,
 Wang and Hudzik in  their seminal work in  \cite{Wang2007}  
 proved NLS ($b=0, \alpha \in 2 \mathbb N$) is GWP for small data in $M^{2,1}$. This enables us to consider the lower-regularity initial data in modulation spaces
  for a class of data out of the critical Sobolev spaces $H^{s_c}$.  See Proposition \ref{exa}. 
Later taking algebra property (Lemma \ref{srp}  \eqref{srpa}) into account, in    \cite{Bhimani2016, Kasso2009},  it is proved that  NLS is LWP  in $M^{p,1}_s \ (1\leq p  \leq \infty, s\geq 0)$ and in $M^{p,q}_{s}\ (1\leq p, q \leq \infty, s>\frac{n}{q'})$ via fixed point argument.   Guo \cite{guo20171d} proved  1D cubic NLS is LWP in 
$M^{2,p} \ (2\leq p < \infty)$ and  later  Oh-Wang \cite{oh2018global}  established global existence  for this result. Chaichenets et al. in \cite{LeonidIn,leonidthesis} established  global well-posesness for cubic NLS in $M^{p, p'}$ for $p$ sufficiently close to 2.

On the other hand,  Bhimani et al. in \cite{BhimaniHartree-Fock, bhimani2023mixed} established GWP in  $M^{p,q}\cap L^2$ for Hartree equation (and also for Hartree-Fock eq.),  i.e. NLS with nonlinearity  $H_b(u)=(|\cdot|^{-b}\ast |u|^2)u$.  The key  step in their work is to get the trilinear estimate for $H_b(u)$ and  this was possible due to the regularizing effect of convolution. However,  it seems their method does not apply to treat nonlinearity $|\cdot|^{-b}|u|^{\alpha}u$ (due to  just multiplication by singular potential).  We also  mention  that, for $M^{p,q}_s$ with some negative regularity $(s<0)$,    strong ill-posedness theory (norm inflation with infinite loss of regularity) is developed   for NLS and Hartree equations  in  \cite{BhimaniNorm, BhiStroIll}.

In spite of these progresses   and ongoing interest for adapting modulation spaces in dispersive PDEs,  there are no results for  INLS \eqref{INLS} so far. In view of this and the   aforementioned  discussion of the previous section (see Remark \ref{sr}), we are inspire to study INLS \eqref{INLS} in modulation spaces. 

\subsection{Main results} We are now ready to state  LWP result in the following theorem. To this end, we denote  
\begin{equation}\label{bforglobal}
\tilde{b}=
      \begin{cases}
          \frac{3-\sqrt{7}}{2} \quad &\text{if} \quad n=1 \\
          2-\sqrt{2} \quad &\text{if} \quad n=2\\
          \frac{n+6-\sqrt{(n+6)^{2}-32}}{4} &\text{if} \quad n\geq 3.
      \end{cases}
\end{equation}
\begin{thm}[Local well-posedness]\label{lwp} Let $0<\alpha<\frac{4-2b}{n},$ $\tilde{b}$ be as in \eqref{bforglobal} and
\begin{eqnarray*}
 \begin{cases}
     0<b\leq \tilde{b} \quad if \  n \neq 2\\
     0<b< \tilde{b}  \quad  if  \  n=2.
 \end{cases}   
\end{eqnarray*}
  Assume that $u_{0}\in L^2+M^{\alpha+2,(\alpha+2)'}.$ Then there exists $T^*=T^*(\|u_{0}\|_{L^2+M^{\alpha+2,(\alpha+2)'}},n,\alpha)>0$ and a unique maximal  solution $u$ of INLS \eqref{INLS} such that 
    \begin{equation*}
    u\in C([0,T^*),L^2)~\cap~ L^{\frac{4(\alpha+2)}{n\alpha}}([0,T^*),L^{\alpha+2}) ~+~C([0,T^*),M^{\alpha+2,(\alpha+2)'}).
     \end{equation*}
     Moreover, 
     \begin{enumerate}
     \item  (Blow-up alternative) If  $T^*<\infty,$ then  $\|u(\cdot,t)\|_{L^2+M^{\alpha+2,(\alpha+2)'}} 
     \to \infty$ as $t \to T^{*}$.
      \item (Lipschitz continuity)
The mapping $u_0\mapsto u(t)$ is locally Lipschitz from $L^2+M^{\alpha+2,(\alpha+2)'}$ to $C([0,T'],L^2)\cap L^{\frac{4(\alpha+2)}{n\alpha}}([0,T'],L^{\alpha+2}) +C([0,T'],M^{\alpha+2,(\alpha+2)'})$ for $T'<T^{*}.$
     \end{enumerate}
      \end{thm}
 Up to now we cannot solve INLS \eqref{INLS} in $L^p_s-$Sobolev spaces or in $L^p$ but in $M^{p, p'}$ for $p \neq 2$. See Remark \ref{sr}, \eqref{use}  and \eqref{se}. Thus, Theorem \ref{lwp} complements the points mentioned in Remark \ref{sr}. 
\begin{Remark}
     We briefly mention key ideas to prove Theorem \ref{lwp}.
\begin{enumerate}
    \item[-]
     By utilizing the boundedness of the operator $e^{it\Delta}$ in modulation spaces (Proposition \ref{mpq}), we treat  the linear part.
    \item[-]  The main difficulty is to handle the spatially decaying factor $|\cdot|^{-b}$ in the nonlinearity as $|\cdot|^{-b}$ does not belong to any $L^{p}$ spaces. We handle this by decomposing $\mathbb R^n=B \cup B^{c}$ into two parts.  Specifically,  
    \begin{equation*}
    \begin{cases}
       |\cdot|^{-b}\in L^{\gamma}(B) & if  \  \frac{n}{\gamma} - b > 0\\
       |\cdot|^{-b}\in L^{\gamma}(B^c) & if  \  \frac{n}{\gamma} - b < 0.
    \end{cases}
    \end{equation*}
\item[-] This introduces additional complexity as we need to select different admissible pairs  for the regions inside the ball i.e. in $B$ and outside the ball i.e. in $B^c$. The restriction on $\alpha$ and $b$ comes due to Lemma \ref{lemlwp}, see Remark \ref{rho2}.
\item[-] We invoke  Strichartz estimates (Theorem \ref{SE1}) to run the fixed point argument via Banach contraction principle.  To this end, we used embedding $M^{p, p'}\subset L^p,$ which also justifies the selection of the exponent $p = \alpha + 2$.
\item[-]We shall see that the life span of the solution $T$ depends on the size of the initial data $u_{0}$ in the $L^2+M^{\alpha+2,(\alpha+2)'}$ norm since the problem is ``subcritical''. 
\end{enumerate}   
\end{Remark}

The local solution established in Theorem \ref{lwp} can be extended to a global one under certain restriction on exponent $p$. 
To this end, we denote
\begin{equation*}\label{pmax}p_{\max} := 
\begin{cases}
 \frac{4\alpha+8-n\alpha}{2\alpha+2+b+\frac{n(4-2b-n\alpha)}{2(\alpha+2)(n+2-b)}}   \quad &\text{if} \  \alpha-\frac{n\alpha^2}{4(\alpha+2)}-\frac{4-2b-n\alpha}{4}+\frac{n(4-2b-n\alpha)}{4(\alpha+2)(n+2-b)}>0 \\ \alpha+2 \quad  & \text{otherwise}.  
\end{cases}
\end{equation*}}
We are now ready to state GWP result in the next theorem.

\begin{thm}[Global well-posedness]\label{gwp}  Let $0<\alpha<\frac{4-2b}{n}, \tilde{b}$ be as in \eqref{bforglobal} and
\begin{eqnarray*}
 \begin{cases}
     0<b\leq \tilde{b} \quad if \  n \neq 2\\
     0<b< \tilde{b}  \quad  if  \  n=2.
 \end{cases}   
\end{eqnarray*} Assume that  $u_{0}\in M^{p,p'}$ for $p\in (2,p_{max}).$ Then  INLS \eqref{INLS} has a unique solution $u$ satisfying 
\begin{equation*}
    u\in C(\R,L^2) \cap L^{\frac{4(\alpha+2)}{n\alpha}}_{loc}(\R,L^{\alpha+2})+ C(\R,M^{\alpha+2,(\alpha+2)'}).
\end{equation*}
\end{thm}
 Theorem \ref{gwp}  is the first GWP result for INLS \eqref{INLS} in modulation spaces.
 We note that  existing results does not address GWP in \( H^s \) for \( 0 < s < 1 \) in the mass-subcritical case without any restriction on the initial data being small. See Remark \ref{sr}.
  Since $H^s \subset M^{p,p'}$ for some $s\in (0,1)$ (see \eqref{use}), Theorem \ref{gwp} partially addresses this gap and complement known results for INLS \eqref{INLS}.

  In \cite{Bourgain1999},  Bourgain  introduced a general scheme, so called Bourgain's high-low decomposition method to establish GWP for NLS in $H^{s}$
for $s>\frac{3}{5}$. The idea is to split initial data $u_0$ between two suitable function spaces and solve in each of them a different NLS. And then combine the solutions to get a function that solves original problem. See Remark \ref{bsgp}, \cite[Section 3.2]{KenigonBourgain} and \cite[Section 3.9]{TaoBook} for more details. Later, Vargas and Vega \cite{vargas2001global} adapted this method to establish global existence for 1D cubic NLS even if $\|u_0\|_{L^2}=\infty$. See \cite{hyakuna2012existence, Kenig2000} for more related works. Recently, Chaichenets et al. in  \cite[Theorem 3]{LeonidIn} and \cite{leonidthesis} successfully  adapted this method to establish GWP for NLS (i.e. INLS \eqref{INLS} with $b=0$) 
in modulation spaces.
Our method of the proof is inspired by  these results.
\begin{Remark}
    \label{bsgp} We briefly mention key ideas to prove Theorem \ref{gwp}.
   \begin{enumerate}
       \item[-] We split the initial data into two parts using Lemma \ref{ipt}: 
       \[u_0=\phi_0 + \psi_0 \in L^2 + M^{(\alpha+2), (\alpha +2)'}.\]
       We say part in $L^2$ having high frequency in the sense that $\|\phi_0\|_{L^2} \lesssim_{p} N^{\beta}$ for some $\beta>0$ and part in $M^{(\alpha+2), (\alpha +2)'}$ having low frequency in the sense that $\|\psi_0\|_{M^{(\alpha+2), (\alpha+2)'}} \lesssim_{p} N^{-1}.$ See \eqref{dp} and \eqref{asi}.
       \item[-] By  Theorem \ref{lwp}, we get   local existence in  $L^2 + M^{(\alpha+2), (\alpha +2)'}$. Guzm\'an \cite[Theorem 1.8]{guzman2020} proved INLS \eqref{INLS} is GWP  in $L^2.$ This  proof depends on the fact that  the solution enjoys the conservation of mass (stated in \eqref{mass} and \eqref{gwpl2}). However,  there is no known useful conservation law for the solution in terms of $M^{p,q}$-spaces norm that help  us to extend the solution globally in time. 
       \item[-] In order to handle this situation, we construct a solution  $u$ of INLS \eqref{INLS} in the form \[u= (v_{0} + w_{0}) + e^{it\Delta} \psi_0.\] 
       Here $v_{0} $ is the $L^{2}$- global solution of \eqref{INLS} with data $\phi_0$, see \eqref{gwpl2}. While $e^{it\Delta}\psi_0 \in M^{(\alpha +2), (\alpha+2)'}$ for $t\in \R$ is the linear evolution of $\psi_0,$ see Proposition \ref{mpq} and $w_{0}$ is the nonlinear interaction term; see \eqref{w01}. See \eqref{solnlocal} for more details.
       \item[-] Since $\|\psi_0\|_{M^{(\alpha+2), (\alpha+2)'}}\lesssim_{p} N^{-1}$ can be made small, we get $v_{0} + w_{0}$ close to $v_{0}$ in $L^2.$ Consequently, $M(v_{0} + w_{0}),$ although no longer conserved, grows slowly enough to yield  a global solution. See Remark \ref{comments} for more details on the proof of Theorem \ref{gwp}.
   \end{enumerate} 
\end{Remark}
\begin{Remark} The restriction on $p$ essentially  comes due to  our choice of suitable  norm size of the high frequency part of  initial data.  We refer to  Remark \ref{Whypso} for  details.
\end{Remark}

\begin{Remark} Theorem \ref{gwp} may be interpreted as follows: for data   $u_{0} \in M^{p,p'}$,  
the solution $u(\cdot,t)$  lies in a larger modulation spaces 
$$\left(L^2 \cap L^{\alpha +2}\right) + M^{\alpha+2,(\alpha+2)'} \supset M^{p, p'}.$$
We have  this loss of regularity as we rely on decomposition of initial data  and we do not know whether  Schr\"odinger propogator $e^{it \Delta}$ gain any regularity  on $M^{p, q}$ (although it is bounded). See \eqref{solnlocal} and Proposition \ref{mpq}.
\end{Remark}

\begin{Remark}\label{edm} 
In the early 1980s  Feichtinger  \cite{Feih83} introduced a  class of Banach spaces,  the so called modulation spaces, which allow a measurement of space variable and Fourier transform variable of a function or distribution on $\mathbb R^n$ simultaneously using the short-time Fourier transform (STFT).  The  STFT  of a tempered distribution $f \in \mathcal{S}'(\R^n)$ with respect to a window function $0\neq g \in {\mathcal S}(\R^n)$ (Schwartz space) is defined by
 $$V_{g}f(x, \omega) =\int_{\mathbb{R}^n} f(t) \overline{g(t-x)} e^{-2 \pi i \omega. t} \, dt ,$$
whenever the integral exists.  It is known that 
\begin{eqnarray*}
\|f\|_{M^{p,q}_s}\asymp   \left\| \|V_gf(x,\omega)\|_{L^p_x} (1+ |\omega|^2)^{s/2} \right\|_{L_{\omega}^q,}
\end{eqnarray*}
see \cite{Feih83},  \cite[Proposition 2.1]{Wang2007}.
\end{Remark}

\par{    This paper is organised as follows. Preliminaries and notations are introduced in Section \ref{NP}, which will be helpful in the sequel. In Section \ref{seclwp}, we obtain the estimates of the nonlinear term and the proof of the Theorem \ref{lwp} is also discussed. In Section \ref{secgwp}, the proof for Theorem \ref{gwp} is presented.

\section{Notations and Preliminaries}\label{NP}
\noindent
\textbf{Notations}.\noindent The symbol $X \lesssim Y$ means 
 $X \leq CY$ 
for some constant $C>0.$ While $X \approx Y $ means $C^{-1}X\leq Y \leq CX$ for some constant $C>0.$ 
The symbol $A \hookrightarrow B$ denotes the continuous embedding of the topological linear space $A$ into $B.$  
 The norm of the space-time Lebesgue spaces $L^{q}([0,T],L^{r}(\R^{n}))$ is defined as
$$\|u\|_{L^{q}_{T}L^{r}}:=\|u\|_{L^{q}([0,T],L^{r}(\R^{n}))}=\left(\int_{0}^{T} \|u(\cdot,t)\|_{L^{r}(\R^{n})}^{q} dt\right)^{\frac{1}{q}}.$$
We simply write $\|u\|_{L^{q}L^{r}}$ in place of $\|u\|_{L^{q}(\R,L^{r}(\R^{n}))}.$
The unit ball centred at $0$ is denoted by $ B = B(0,1)$ and $B^{c}=\R^{n} \setminus B.$\\
\subsection{Basic Properties of \texorpdfstring{$M^{p,q}$-spaces}{M^{p,q}-spaces}}
Let us briefly recall some relevant facts concerning the modulation spaces.
\begin{lem}[\cite{KassoBook,wang2011harmonic}] \label{srp}
    Let $p,q,p_{i},q_{i} \in [1,\infty], (i=0,1,2)$ and $s_1, s_2 \in \R.$ 
  \begin{enumerate}
     \item $ \label{srp1}  M_{s_1}^{p_{1},q_{1}}\hookrightarrow M_{s_2}^{p_{2},q_{2}}$ whenever $p_{1}\leq p_{2}, q_{1}\leq q_{2},  s_2 \leq s_1.$
    \item \label{srp2} $  M_{s_1}^{p,q_{1}}\hookrightarrow M_{s_2}^{p,q_{2}}$ whenever $ q_{2}<q_{1},  s_1-s_2 > \frac{n}{q_2}-\frac{n}{q_1}$. 
            \item\label{embedding} $M^{p,q_1}\hookrightarrow L^{p} \hookrightarrow  M^{p,q_2} $ for $q_1 \leq \min \{p, p'\}$ and $q_2 \geq \max \{p, p'\}.$
            
\item \label{srpa}If $\frac{1}{p_{1}}+\frac{1}{p_{2}}=\frac{1}{p_{0}}$ and $\frac{1}{q_{1}}+\frac{1}{q_{2}}=1+\frac{1}{q_{0}},$ then $\|f g\|_ {M^{p_{0},q_{0}}} \lesssim  \|f\|_ {M^{p_{1},q_{1}}} \|g\|_ {M^{p_{2},q_{2}}}.$
    \end{enumerate}
    \end{lem}
\begin{prop}[examples]\label{exa}Let $1\leq p, q \leq \infty,$ $s_1, s_2 \in \R.$  
\begin{enumerate}
\item \label{e1}  The $M^{p,q}_s-$spaces are invariant under Fourier transform  only when $p=q$. See for e.g. \cite[p. 2084]{Ruzet4}. It follows that $M^{p,p'} \neq H^{s}=M^{2,2}_s$ for $p \neq 2, s\in \mathbb R.$ 
\item \label{e2} For $p>2$ and $s_b<0,$  $H^{s_b}\not\subset M^{p,p'}$. 
\item \label{e3} $M_{s_c}^{2,1}\subset H^{s_c}$ and $M_s^{2,1}\not\subset H^{s_c}$ for $s<s_c=s_0.$  
\item (See \cite{kobayashi2011inclusion}).Denote 
$$\tau (p,q)= \max \left\{ 0, n\left( \frac{1}{q}- \frac{1}{p}\right), n\left( \frac{1}{q}+ \frac{1}{p}-1\right) \right\}.$$Then 
$L^{p}_{s_1} \subset M^{p,q}_{s_2}$ if and only if one of the following conditions is satisfied:
\begin{gather*}
 (i) \  q\geq p>1, s_1\geq s_2 + \tau(p,q); \quad (ii) \ p>q, s_1>s_2+ \tau(p,q);\\
(iii) \ p=1, q=\infty, s_1\geq s_2 + \tau(1, \infty); \quad (iv) \ p=1, q\neq \infty, s_1>s_2+\tau (1, q).
\end{gather*}
\item For completely sharp embedding between $M^{p,q}_{s_1}$ and $H^s,$ see  \cite{kobayashi},  \cite[Theorems 1.3 and 1.4]{kato}.
\end{enumerate}
 \end{prop}
\begin{proof}
In fact, if $H^{s_b}\subset M^{p,p'}$, then by Lemma \ref{srp}\eqref{embedding}, we have $L^2 \subset L^p,$ which is a contradiction. For part \eqref{e3}, see e.g.  \cite[Section 1]{RWZexe},  \cite[Section 3]{Wang2006},   \cite{kobayashi, wang2011harmonic}.
\end{proof}  
\begin{prop}[See Proposition 4.1 in \cite{Wang2007}, \cite{KassoBook}]\label{mpq} Let $ p,q\in [ 1,\infty]$ and $ s \in \R.$ Denote the 
Schr\"{o}dinger propagator by
\begin{equation*}
    e^{it\Delta}f(x):= \int_{\R^{n}} e^{i\pi t|\xi|^{2}} \hat{f}(\xi) e^{2\pi i  \xi \cdot x}d \xi \quad (f\in \mathcal{S}, t \in \mathbb R).
\end{equation*}
\begin{enumerate}
    \item \label{spe} $\|e^{it\Delta}f\|_{M_s^{p,q}} \lesssim_{n,s}(1+t^2)^{-\frac{n}{2}(\frac{1}{2}-\frac{1}{p}) }\|f\|_{M_s^{p',q}}\quad (2\leq p \leq \infty).$
    \item  $ \|e^{it\Delta}f \|_{M_s^{p,q}}\lesssim_{n,s} (1+t^{2})^{ \frac{n}{2}\left| \frac{1}{2}-\frac{1}{p} \right|} \|f\|_{M_s^{p,q}}.$
\end{enumerate}
\end{prop}}
\begin{lem}[see e.g. Lemma 3.9 in {\cite{leonidthesis}} ] \label{eg} Denote  \begin{equation*} G(u,v,w)=|u+v|^{\alpha}(u+v)-|u+w|^{\alpha}(u+w)
\end{equation*}
for  $\alpha>0$ and $u,v,w\in \C.$ Then 
     $$|G(u,v,w)| \lesssim_{\alpha} (|u|^{\alpha}+|v|^{\alpha}+|w|^{\alpha})|v-w|.
    $$
  \end{lem}
  We simply write $G(u,v,0)=G(u,v).$
    \begin{lem}[Interpolation; see Theorem 6.1 (D) in \cite{Feih83} and Proposition 5.1 in {\cite{leonidthesis}}]\label{ipt}
     Let $u\in M^{p,p'}, \ p\in (2, r)$ and $N>0.$ Then there exists $v \in L^{2}$ and $w \in M^{r,r'}$  such that $$u= v + w$$ with the following property:
\begin{gather*}
\begin{cases}
 \|v\|_{L^{2}}\leq C \|u\|
        _{M^{p,p'} }N^{\beta}\\ 
        \|w\|_{M^{r,r'}}\leq C \|u\|_{M^{p,p'}}\frac{1}{N} 
\end{cases}, \ \ where  \   \beta = \frac{\frac{1}{2} - \frac{1}{p}}{\frac{1}{p} - \frac{1}{r}}.
\end{gather*} \end{lem}
\begin{Definition}\label{apairs}
 We say that a pair $(q,r)$ is {\it admissible} if $$\frac{1}{q} = \frac{n}{2} \left( \frac{1}{2} - \frac{1}{r} \right) \quad
\text{where} \begin{cases} 
2 \leq r \leq \frac{2n}{n-2} & \text{if } n \geq 3, \\
2 \leq r < \infty & \text{if } n = 2, \\
2 \leq r \leq \infty & \text{if } n = 1.
\end{cases}$$
\end{Definition}
The set of all admissible pairs is denoted by $$\mathcal{A}= \{(q,r):(q,r) \; \text{is admissible pair} \}.$$
    \begin{thm}[Strichartz estimates, \cite{linares2020,cazenave2003semilinear,KeelTao1998}]\label{SE1}
Let $f \in L^2,$ $(\Tilde{q},\Tilde{r}) ,(q,r) \in \mathcal{A}$ and $g \in L^{\Tilde{q}'} 
(\R ,\, L^{\Tilde{r}'}  )$.
Then
\Bea \| e^{-it\Delta} f \|_ {L^{q}
L^{r}  }&\lesssim_{n,r}&
\|f \|_{L^{2} } \\
 \left\| \int_0^t e^{-i(t-s)\Delta} g(x,s) ds 
\right\|_{L^{q}  L^{r}  } &\lesssim_{n,r,\Tilde{r}}&  \|g \|_{L^{\Tilde{q}'} L^{\Tilde{r}'}  }.
 \Eea
\end{thm}

\section{Local well-posedness in $L^{2}+M^{\alpha+2,(\alpha+2)'}$ }\label{seclwp}
In this section, we shall prove Theorem \ref{lwp}. To this end, we shall briefly introduce some notations and prove Lemma \ref{lemlwp}.

Consider the Banach space $X(T)$ expressed as 
\begin{equation}
\begin{aligned}\label{X(T)}
        X(T)&:= X_{1}(T)+X_{2}(T)
\end{aligned}
\end{equation}
where $$X_1(T):=C([0,T],L^2) ~\cap~ L^{\frac{4(\alpha+2)}{n\alpha}}([0,T],L^{\alpha+2})$$
equipped with the norm 
$$\|v\|_{X_1(T)}=\max \left\{\|v\|_{L^{\infty}_{T}L^2},\|v\|_{L^{\frac{4(\alpha+2)}{n\alpha}}_{T}L^{\alpha+2}}\right\} $$ and $$ X_2(T):=C([0,T],M^{\alpha+2,(\alpha+2)'}).$$\\
The norm on $X(T)$ is given as 
\begin{equation*}
\|u\|_{X(T)} =\inf_{\substack{u=v+w \\ v \in X_1(T) \\ w \in X_2(T)}} \left(\|v\|_{X_1(T)} + \|w\|_{X_2(T)} \right).
\end{equation*}
Denote
\begin{equation}\label{YT}
 Y(T) := L^{\frac{4(\alpha+2)}{n\alpha}}([0,T],L^{\alpha+2}).
\end{equation}

 The  following lemma  provides an estimate for the nonlinearity $|x|^{-b} |u|^{\alpha} u$ which will play a crucial role in our analysis. 
\begin{lem}\label{lemlwp}
    Let $0<\alpha<\frac{4-2b}{n},~\tilde{b}$ be as in \eqref{bforglobal} and
\begin{eqnarray*}
 \begin{cases}
     0<b\leq \tilde{b} \quad if \  n \neq 2\\
     0<b< \tilde{b}  \quad  if  \  n=2.
 \end{cases}   
\end{eqnarray*}Then
$$\inf_{(\gamma,\rho)\in \mathcal{A}}\| ~|x|^{-b}|u|^{\alpha}v\|_{L^{\gamma'}_{T}L^{\rho '}} \lesssim  ( T^{1-\frac{n\alpha}{4}}+ T^{\frac{4-2b-n\alpha}{4}-\frac{n(4-2b-n\alpha)}{4(\alpha+2)(n+2-b)}})\|u\|_{Y(T)}^{\alpha}\|v\|_{Y(T)}.$$
   
\end{lem}

\begin{proof} Recall $B=B(0,1).$ Note  that 
\begin{align*}
    \inf_{(\gamma,\rho)\in \mathcal{A}} \|~|x|^{-b}|u|^{\alpha}v\|_{L^{\gamma'}_{T}L^{\rho '}} &\leq \inf_{(\gamma,\rho)\in \mathcal{A}}\|~|x|^{-b}|u|^{\alpha}v\|_{L^{\gamma'}_{T}L^{\rho '}(B^{c})} +\inf_{(\gamma,\rho)\in \mathcal{A}}\|~|x|^{-b}|u|^{\alpha}v\|_{L^{\gamma'}_{T}L^{\rho '}(B)} \\
    & = A_{1}+A_{2}.
    \end{align*}
    We need to find an admissible pair $(\gamma_{1},\rho_{1})$ to estimate $A_{1}$ by $Y(T)$ norm. Using H\"older's inequality twice, we obtain
    \begin{align}
        A_{1} &\lesssim \|~|u|^{\alpha}v\|_{L^{\gamma'_{1}}_{T}L^{\rho '_{1}}(B^{c})} \nonumber\\
        &\label{YTB1}\leq T^{1-\frac{n\alpha}{4}}\|u\|_{Y(T)}^{\alpha}\|v\|_{Y(T)} 
    \end{align}
 with $\rho'_{1}$ satisfying H\"older conditions
  \begin{equation}\label{H1x}
      \frac{1}{\rho'_{1}}=\frac{\alpha}{\alpha+2}+\frac{1}{\alpha+2}.
\end{equation}

  This gives value of $\rho'_{1}=\frac{\alpha+2}{\alpha+1} $  (and $\rho_{1}=\alpha +2)$. Since $(\gamma_{1},\rho_{1})$ is an admissible pair \footnote{It is easy to check that $\rho_{1}=\alpha +2$ satisfies all conditions in Definition \ref{apairs} for the bound on $\alpha.$}, we get $$\gamma'_{1}=\frac{4(\alpha+2)}{4(\alpha+2)-n\alpha}.$$ 
  While $\gamma'_{1}$ satisfies H\"older conditions \footnote{Each H\"older exponent in \eqref{H1x} and \eqref{H1t} lies in the interval $[1,\infty]$ by the hypothesis on $\alpha$.} 
  \begin{equation}\label{H1t}
  \frac{1}{\gamma'_{1}}=\frac{1}{\frac{4(\alpha+2)}{4(\alpha+2)-n\alpha}}=\frac{1}{\nu}+\frac{\alpha}{\frac{4(\alpha+2)}{n\alpha}}+\frac{1}{\frac{4(\alpha+2)}{n\alpha}}.
  \end{equation}
Here $\frac{1}{\nu}$ represents the exponent of $T.$ Solving for $\frac{1}{\nu}$ \footnote{Note that $\frac{1}{\nu}$ is positive by the hypothesis on $\alpha$.} , we get $$\frac{1}{\nu}=1-\frac{n\alpha}{4}.$$

Similarly, we need to find an admissible pair $(\gamma_{2},\rho_{2})$ to estimate $A_{2}$ by $Y(T)$ norm. Applying H\"older's inequality twice, we obtain
   \begin{align}
        \hspace{1.6cm}A_{2}&\leq \|~|x|^{-b}|u|^{\alpha}v \|_{L^{\gamma'_{2}}_{T}L^{\rho '_{2}}(B)} \nonumber\\
        &\label{A2E1}\leq \| ~\| |x|^{-b}\|_{L^{\gamma_3}(B)} \|u\|^{\alpha}_{L^{\alpha+2}}\|v\|_{L^{\alpha+2}}\|_{L^{\gamma'_{2}}_{T}} \\
        &\label{A2E2}\leq  T^{\frac{1}{q_{1}}} \|~|x|^{-b}\|_{L^{\gamma_{3}}(B)} \|u\|^{\alpha}_{Y(T)}\|v\|_{Y(T)}.
    \end{align}
Here $\rho'_{2},\gamma'_{2}, q_{1}$ and $\gamma_{3}$ satisfies following conditions:
    \begin{eqnarray}
       \label{HXV} \frac{1}{\rho'_{2}} &=& \frac{1}{\gamma_{3}}+\frac{\alpha}{\alpha+2}+\frac{1}{\alpha+2}\\
        \label{HTV}\frac{1}{\gamma'_{2}} &=& \frac{1}{q_{1}}+\frac{n \alpha^{2}}{4(\alpha+2)}+\frac{n\alpha}{4(\alpha+2)}\\
        \label{AP}\frac{1}{\gamma'_{2}} &=&1-\frac{1}{2}\left(\frac{n}{\rho'_{2}}-\frac{n}{2}\right) \\
        \label{tpower}\frac{1}{q_{1}}& >&0 \\
       \label{SLB} \frac{n}{\gamma_{3}}&>&b.
         \end{eqnarray}
        Note that \eqref{HXV} and \eqref{HTV} is due to H\"older's inequality for space and time variable applied to get \eqref{A2E1} and \eqref{A2E2} respectively \footnote{The  H\"older exponents in \eqref{HXV} and \eqref{HTV} lies in the interval $[1,\infty]$ due to the bound on $\alpha$ and $b$.}. Since $(\gamma_{2},\rho_{2})$ is an admissible pair,  we have \eqref{AP}. The exponent of $T$ needs to be positive in \eqref{A2E2}, hence we have \eqref{tpower}. Condition \eqref{SLB} is required as $\|~|x|^{-b}\|_{L^{\gamma_{3}}(B)} <\infty$ if and only if $\frac{n}{\gamma_{3}}>b.$
         \\From \eqref{HXV} and \eqref{SLB},
\begin{equation}
    \frac{n}{\rho'_{2}}-\frac{n(\alpha+1)}{\alpha+2}>b.
\end{equation}
Solving for $\alpha$\footnote{Note that $0 <\frac{2n-2b\rho'_{2}-n\rho'_{2}}{n\rho'_{2}-n+b\rho'_{2}}$ for $b<2$ which holds by the bound on $b.$},
\begin{equation}
    \alpha <\frac{2n-2b\rho'_{2}-n\rho'_{2}}{n\rho'_{2}-n+b\rho'_{2}}.
\end{equation}
By our hypothesis, $0<\alpha<\frac{4-2b}{n}.$ Solving for $\rho'_{2}$ 
\begin{equation*}
    \frac{2n-2b\rho'_{2}-n\rho'_{2}}{n\rho'_{2}-n+b\rho'_{2}}=\frac{4-2b}{n}
\end{equation*} yields 
\begin{equation}
    \rho'_{2}=\frac{2n(n+2-b)}{n(n+4)+2b(2-b)}.
\end{equation}
Inserting the value of $\rho'_{2}$ in \eqref{AP}, we get
\begin{equation}
    \gamma'_{2}=\frac{4(n+2-b)}{(2-b)(n+4-2b)}.
\end{equation}
 Taking Remark \ref{rho2} into account, we have   
  $$(\gamma_{2},\rho_{2}) =\left(\frac{4(n+2-b)}{2n+4b+bn-2b^2} , \frac{2n(n+2-b)}{n^2-2nb-4b+2b^2}\right)\in \mathcal{A}.$$
From \eqref{HTV}, we obtain
\begin{align*}
    \frac{4}{q_{1}}-(4-2b)&=\frac{(2-b)(n+4-2b)}{n+2-b}-\frac{n\alpha(\alpha+1)}{\alpha+2}-(4-2b)\\
    &=\frac{-n(4-2b-n\alpha)}{(n+2-b)(\alpha+2)}-n\alpha.
    \end{align*}This gives
\begin{equation}\label{q1}
    \frac{1}{q_{1}}=\frac{4-2b-n\alpha}{4}-\frac{n(4-2b-n\alpha)}{4(n+2-b)(\alpha+2)}.
\end{equation}
Since $0<\alpha< \frac{4-2b}{n}$ and $ 0<b< \min\{2,n\}$,  note that  $ \frac{1}{q_{1}} $ in \eqref{q1} is a positive quantity.
Substituting the value of $\frac{1}{q_{1}}$ in \eqref{A2E2}, we have \begin{align}
    A_{2} &\label{YTB2}\lesssim T^{\frac{4-2b-n\alpha}{4}-\frac{n(4-2b-n\alpha)}{4(\alpha+2)(n+2-b)}} \|u\|^{\alpha}_{Y(T)}\|v\|_{Y(T)}.
    \end{align}
    Combining \eqref{YTB1} and \eqref{YTB2}, we have the claim.    
\end{proof}

 \begin{Remark}\label{rho2} In  order to estimate $A_{2}$ by $Y(T)$ norm in Lemma \ref{lemlwp},  we used below admissible pair $(\gamma_{2},\rho_{2}),$ which imposes  the  restriction on $b.$
\begin{enumerate}
    \item Under the hypothesis of Lemma \ref{lemlwp}, we have 
    $$(\gamma_{2},\rho_{2}) =\left(\frac{4(n+2-b)}{2n+4b+bn-2b^2} , \frac{2n(n+2-b)}{n^2-2nb-4b+2b^2}\right)\in \mathcal{A}.$$
    \item For $n \geq 3,$ we have  
    $$\left(0<b \leq \frac{n+6-\sqrt{(n+6)^{2}-32}}{4}\right) \implies 2\leq \rho_{2} \leq \frac{2n}{n-2}  \footnote{It is easy to compute that $2\leq \rho_{2}$  for $0 < b <2$.}.$$
    \item  When $n=1,2$, we have $ 2 \leq \rho_{2} \leq \infty$ and $2\leq \rho_{2} < \infty$ respectively for $b$ satisfying
    \begin{equation*}
    \begin{cases}
    0<b\leq \frac{3-\sqrt{7}}{2} \quad &\text{if} \quad n=1 \\
    0<b<2-\sqrt{2} \quad &\text{if} \quad n=2.
    \end{cases}
    \end{equation*}
\end{enumerate}
 \end{Remark}
\begin{Remark}\label{lemlwp2}
For $T\leq 1$ and $\alpha>0$, by Lemma \ref{srp}\eqref{embedding}, we have $$ X(T)=C_TL^2 \cap Y(T) +C_TM^{(\alpha+2), (\alpha +2)'} \hookrightarrow 
 Y(T)\footnote{Note that $\|\cdot\|_{Y(T)} \lesssim_{n} \|\cdot\|_{X(T)}.$}.$$
  Thus, we have
    $$\inf_{(\gamma,\rho)\in \mathcal{A}} \|~|x|^{- b}|u|^{\alpha}v\|_{L_T^{\gamma'}L^{\rho '}} \lesssim_{n} ( T^{1-\frac{n\alpha}{4}}+ T^{\frac{4-2b-n\alpha}{4}-\frac{n(4-2b-n\alpha)}{4(\alpha+2)(n+2-b)}})\|u\|_{X(T)}^{\alpha}\|v\|_{X(T)}$$
    where $b$ and $\alpha$ be as in Lemma \ref{lemlwp}.
\end{Remark}

  \begin{Remark}
     For $b=0,$ we only need to evaluate $A_{1}$ in Lemma \ref{lemlwp} as we do not have any singularity in this case. Cf. \cite{leonidthesis,LeonidIn}.
 \end{Remark}
 
 \begin{Remark}\label{narrowb} 
In contrast to $L^2-$GWP result of Guzm\'an (see \eqref{gwpl2}), we need  to narrow the range of $b$, from $0<b<\min \{2,n\}$\footnote{In \cite[Lemma 3.1]{guzman2020}, $\rho_{2}=\frac{4-2b+2n}{n-b}$ with $\rho_{2} \leq \frac{2n}{n-2}$ for $n \geq 3$ and  $\rho_{2}>\frac{2n}{n-b}$ which would be fulfilled only if $b \leq \min\{2,n\}.$ } to $0<b\leq \tilde{b}$. Revisit Remark \ref{rho2} for details.
 \end{Remark}
 
\begin{proof}[\textbf{Proof of Theorem \ref{lwp}}]By Duhamel's principle, INLS \eqref{INLS} is equivalent to the integral equation 
\begin{equation*}
    u(t)=e^{it\Delta}u_{0}+i \mu \int_{0}^{t}e^{i(t-s)\Delta}|x|^{-b}(|u|^{\alpha}u)(s)ds :=\Lambda(u)(t).
\end{equation*}
Let $a$ and $T$ be a positive real numbers (to be chosen later). 
    Define $$B(a,T):=\{u\in X(T):\|u\|_{X(T)}\leq a\}.$$
We will show that $\Lambda$ 
    is a contraction map on $B(a,T).$ 
    Firstly, we consider the linear evolution of $u_{0}$ where $u_{0}=v_{0}+w_{0}\in L^{2}+ M^{\alpha+2,(\alpha+2)'}, v_{0}\in L^{2}$ and $w_{0}\in M^{\alpha+2,(\alpha+2)'}.$ Assume that $T\leq 1$. Using Theorem \ref{SE1} and Proposition \ref{mpq}, we have
    \begin{align}
        \|e^{it\Delta}u_{0}\|_{X(T)}&\leq      \|e^{it\Delta}v_{0}\|_{X_{1}(T)}  +      \|e^{it\Delta}w_{0}\|_{X_{2}(T)}\nonumber\\
        &\lesssim_{n,\alpha}      \|v_{0}\|_{L^2}  + (1+T^{2})^{\frac{n}{2}\left(\frac{1}{2}-\frac{1}{\alpha+2}\right)}     \|w_{0}\|_{M^{\alpha+2,(\alpha+2)'}}\nonumber\\
         &\label{linearu0}\lesssim_{n}      \|u_{0}\|_{L^2  +  M^{\alpha+2,(\alpha+2)'}}.
    \end{align}
    This suggests the choice of
    $\label{achoice}
    a=C(n,\alpha)\|u_{0}\|_{L^2  +  M^{\alpha+2,(\alpha+2)'}} .   $
    Using $X_{1}(T)\hookrightarrow X(T),$ Theorem \ref{SE1} (with $r\in \{\alpha+2,2\}$) and Remark \ref{lemlwp2}, we have
    \begin{align}
    \left|\left| \int_{0}^{t} e^{i(t-\tau)\Delta}|x|^{-b}(|u|^{\alpha}u)(\tau)d\tau \right|\right|_{X(T)}
    &\lesssim \left|\left| \int_{0}^{t} e^{i(t-\tau)\Delta}|x|^{-b}(|u|^{\alpha}u)(\tau)d\tau \right|\right|_{X_{1}(T)} \nonumber\\
   &\lesssim_{n,\alpha,b} \||u|^{\alpha}u\|_{L^{\gamma'_{1}}_{T}L^{\rho '_{1}}(B^{c})}+\|~|x|^{-b} |u|^{\alpha}u\|_{L^{\gamma'_{2}}_{T}L^{\rho '_{2}}(B)}\nonumber\\
    &\label{ipGamma}\lesssim_{n} ( T^{\frac{4-2b-n\alpha}{4}-\frac{n(4-2b-n\alpha)}{4(\alpha+2)(n+2-b)}})\|u\|_{X(T)}^{\alpha+1}.
    \end{align}
    Taking
    \begin{equation}\label{Tchoice}
        T :=\min \left\{ 1,~C(n,\alpha,b) \|u_{0}\|^{-\frac{\alpha}{\frac{4-2b-n\alpha}{4}-\frac{n(4-2b-n\alpha)}{4(\alpha+2)(n+2-b)}}}_{L^2  +  M^{\alpha+2,(\alpha+2)'}} \right\}~,
    \end{equation}
 we combine \eqref{linearu0} and \eqref{ipGamma} to conclude $\Lambda(u)\in B(a, T).$ Similarly,  one can show that $\Lambda(u)$ is a contraction mapping. In fact, 
 using Lemma \ref{eg} for $G(0,u,v)$ and the previous argument employed to bound the integral part, for $u,v \in B(a,T)$, we have
\begin{align}
\hspace{-4cm}\|\Lambda(u)-\Lambda(v)\|_{X(T)} & \lesssim \left|\left| \int_{0}^{t} e^{i(t-\tau)\Delta}|x|^{-b}G(0,u,v)(\tau)d\tau \right|\right|_{X_{1}(T)} \nonumber
\end{align}
\begin{align}
&\lesssim_{n,\alpha,b} \|(|u|^{\alpha}+|v|^{\alpha})|u-v|\|_{L^{\gamma'_{1}}_{T}L^{\rho '_{1}}(B^{c})}+\|~|x|^{-b}(|u|^{\alpha}+|v|^{\alpha})|u-v|\|_{L^{\gamma'_{2}}_{T}L^{\rho '_{2}}(B)} \nonumber \\
    &\lesssim_{n} ( T^{\frac{4-2b-n\alpha}{4}-\frac{n(4-2b-n\alpha)}{4(\alpha+2)(n+2-b)}})(\|u\|_{X(T)}^{\alpha}+\|v\|_{X(T)}^{\alpha})\|u-v\|_{X(T)}.
\end{align}
By the choice of $a$ and $T$ \footnote{In the expression of $a$ and $T, C(n,\alpha,b)$ is choosen sufficiently small to ensure that $\Lambda(u)$ is a contraction map on $B(a,T).$}, we have $$\|\Lambda(u)-\Lambda(v)\|_{X(T)}\leq \frac{1}{2}\|u-v\|_{X(T)}.$$
Thus, by  the contraction mapping theorem,  we  obtain a unique fixed point  for $\Lambda,$ which is a solution to INLS \eqref{INLS}.
 By the standard argument, one can establish the blow-up alternative.\\
 For Lipschitz continuity, let $ v_{0}, w_{0} \in V $, where $ V $ is a neighborhood of $ u_{0}$. Denote by $v$ and $w$ the unique maximal solutions of INLS \eqref{INLS} over the interval $[0,T^*)$ with initial values $ v_{0} $ and $ w_{0} $, respectively. Thus, \( v \) and \( w \in X(T')\) for \( T' < T^{*}.  \)
Hence
\begin{eqnarray*}
    \|v-w\|_{X(T')} &=& \|e^{it\Delta}v_{0}-e^{it\Delta}w_{0} + \Lambda(v) - \Lambda(w) \|_{X(T')}\\
    & \lesssim_{n,\alpha} & \|v_{0}-w_{0}\|_{L^2+M^{\alpha+2,(\alpha+2)'}}+\frac{1}{2}\|v-w\|_{X(T')}.\\
\text{Thus,}\;
\|v-w\|_{X(T')} &\lesssim_{n,\alpha}& \|v_{0}-w_{0}\|_{L^2+M^{\alpha+2,(\alpha+2)'}}
\end{eqnarray*}
for any $v_{0}, w_{0} \in V \subset L^2 + M^{\alpha+2,(\alpha+2)'}.$ This concludes local Lipschitz continuity.
\end{proof}
\section{Global well-posedness in $M^{p, p'}$}\label{secgwp}
In this section, we shall prove Theorem \ref{gwp}.  We start by decomposing (using Lemma \ref{ipt}) initial data $u_0 \in M^{p, p'} \subset L^2 + M^{(\alpha +2), (\alpha +2)'} $ into two parts such that the size of $M^{(\alpha +2), (\alpha +2)'}-$data can be controlled by  arbitrary small quantity. Specifically, for any $N>1$  and given $u_0 \in M^{p, p'},$ there exists  $\phi_0 \in L^2, \psi_0 \in M^{(\alpha +2), (\alpha +2)'}$ (depending on $N$)   such that 
\begin{equation}\label{dp}
    u_0= \phi_0 + \psi_0
\end{equation}
with 
\begin{eqnarray}\label{asi}
\|\phi_0\|_{L^2} \lesssim N^{\beta},  \quad \|\psi_0\|_{M^{(\alpha +2), (\alpha +2)'}} \lesssim  \frac{1}{N}
\end{eqnarray}
where\begin{equation}\label{betap}
    \beta = \frac{\frac{1}{2} - \frac{1}{p}}{\frac{1}{p} - \frac{1}{\alpha+2}}.
\end{equation}
Firstly, we consider  INLS \eqref{INLS} with initial data $\phi_0:$
\begin{eqnarray}\label{ivpL2}
\begin{cases}
 i \partial_t v_{0} + \Delta v_{0}+\mu |x|^{-b}|v_{0}|^{\alpha}v_{0}=0 \\
v_{0}(\cdot,0)=\phi_{0}\in L^2.   
\end{cases}        
    \end{eqnarray}
We recall that, in  \cite[Theorem 1.8]{guzman2020},  Guzm\'an proved that \eqref{ivpL2} has a unique global solution 
\begin{eqnarray}\label{gwpl2}
  v_{0}\in C(\R, L^2) \cap L_{loc}^{q}(\R,L^{r} )
\end{eqnarray}
\begin{equation}\label{v0qr2}
\sup_{(q,r)\in \mathcal{A}}\|v_{0}\|_{L_{loc}^{q} L^{r}} \lesssim_{n,r} \|\phi_{0}\|_{L^{2}}.
\end{equation}
Now, consider the modified INLS \eqref{INLS} associated with the evolution of $\psi_{0}$: 
\begin{eqnarray}\label{ivpMod}
\begin{cases}
 i \partial_t w + \Delta w +\mu |x|^{-b}(|w+v_{0}|^{\alpha}(w+v_{0})-|v_{0}|^{\alpha}v_{0})=0 \\
w(\cdot,0)=\psi_{0}\in M^{\alpha+2, (\alpha+2)'} .
\end{cases}
\end{eqnarray}
The solution of above I.V.P \eqref{ivpMod} is given as
\begin{equation}\label{modsoln1}
w =e^{it\Delta}\psi_{0}+w_{0}.
\end{equation}
The  nonlinear interaction $w_0$ corresponding to $\psi_0$ can be expressed as \begin{align}\label{modsoln2}
    w_{0} = i \mu \int_{0}^{t}  e^{i(t-\tau)\Delta}|x|^{-b} \left( \big| w+v_{0}\big|^{\alpha} (w+v_{0}) - |v_{0}|^{\alpha} v_{0}\right)(\tau)  \, d\tau.
\end{align}
Formally,  we may rewrite  the  solution to  INLS \eqref{INLS} corresponding to data $u_{0}$ in \eqref{dp} as follows \begin{eqnarray}
u=v_{0}+w = v_{0}+e^{it\Delta}\psi_{0}+w_{0}.\label{solnlocal}
\end{eqnarray}
 In view of \eqref{gwpl2} and Proposition \ref{mpq},   we notice that  $v_0$ and $e^{it\Delta} \psi_0$ are globally defined in appropriate spaces.  In order to establish desire  global existence,  we  first understand  the time interval of existence for $w_0$.   To this end, we may rewrite
\begin{align}\label{w01}
    w_{0} 
    =i \mu \int_{0}^{t} e^{i(t-\tau)\Delta}|x|^{-b} G(v_{0} + e^{i\tau\Delta}\psi_{0}, w_{0})(\tau) \, d\tau +  i \mu \int_{0}^{t} e^{i(t-\tau)\Delta}|x|^{-b} G(v_{0}, e^{i\tau\Delta}\psi_{0})(\tau) \, d\tau,
\end{align}
 with $G$ as defined in Lemma \ref{eg}. And more generally, we have the following local well-posedness result for   perturb integral equation \eqref{w01}.

\begin{prop}\label{w0exist} Let $\phi \in L^2, \psi\in M^{(\alpha +2), (\alpha +2)'}$ and $Y(T)$ be as in  \eqref{YT}.  Assume that  $b$ and $\alpha$  be  as in Theorem \ref{gwp}. Denote by  $v$  the $L^2-$global solution  (as in \eqref{gwpl2}) for initial value $\phi$.
Then  there exists  a constant $C=C(n, \alpha,b)>0$ such that  integral equation $$w= i \mu \int_{0}^{t} e^{i(t-\tau)\Delta}|x|^{-b} G(v + e^{i\tau\Delta}\psi, w)(\tau) \, d\tau +  i \mu \int_{0}^{t} e^{i(t-\tau)\Delta}|x|^{-b} G(v, e^{i\tau\Delta}\psi)(\tau) \, d\tau$$ has a unique solution  $w \in Y(T)$   provided $T$ satisfying
\begin{align}
    \label{c1} T &\leq 1 \\
    \label{c2} T &\leq C \left( \|\phi\|_{L^{2}} + \|\psi\|_{M^{\alpha+2,(\alpha+2)'}} \right)^{-\frac{\alpha}{\frac{4-2b-n\alpha}{4} - \frac{n(4-2b-n\alpha)}{4(\alpha+2)(n+2-b)}}} \\
    \label{c3} T &\leq C \left( \|\psi\|_{M^{\alpha+2,(\alpha+2)'}} \right)^{-\frac{\alpha}{\frac{4-2b-n\alpha}{4} - \frac{n(4-2b-n\alpha)}{4(\alpha+2)(n+2-b)} + \frac{n\alpha^{2}}{4(\alpha+2)}}}.
\end{align}
\end{prop}
\begin{Remark} \label{nexp}
   Since $0<\alpha<\frac{4-2b}{n}$ and $0<b<\tilde{b}$ \footnote{Note that $\tilde{b} < \min\{2,n\}$ for all $n$.}, it follows that $\frac{n}{(\alpha+2)(n+2-b)}<1$. Thus, the exponents on right hand side of conditions \eqref{c2} and \eqref{c3} involving $L^{2}$ and  $M^{\alpha+2,(\alpha+2)'}$- norms are both negative.
\end{Remark}
 \begin{cor}[growth of perturb solution]\label{winfty2}
 Under the hypothesis of Proposition \ref{w0exist}, there exists a constant $C(n,\alpha,b)$ satisfying
    \begin{equation*}
\|w\|_{L^{\infty}_{T}L^{2}}\lesssim_{n,\alpha,b} T^{\frac{n\alpha}{4(\alpha+2)}}\|\psi\|_{M^{\alpha+2,(\alpha+2)'}}.
    \end{equation*}

 \end{cor}
We postpone the proof of Proposition \ref{w0exist} and Corollary \ref{winfty2} for the moment.  And once we have these tools, we are ready to prove Theorem \ref{gwp}.  We start with the following remark. 
\begin{Remark}[proof strategy]
     The following key points are in order to establish a global existence.
     \begin{itemize}\label{comments}
          \item[-] Revisit Remark \ref{bsgp}. In view of the discussion in the para below \eqref{solnlocal}, we are left to handle nonlinear interaction $w_0$. This we shall do using Proposition \ref{w0exist} and Corollary \ref{winfty2}.
          \item[-]   To this end, taking $\phi=\phi_{0}$ and $\psi=\psi_{0}$, note that the solution of $ w_{0}$ exists on  time-interval  $[0,T(N)]$ with $T=T(N)$ 
          satisfying conditions \eqref{c1} to \eqref{c3}.
          \item[-] We extend this time-interval by the iterative  scheme, which involve the following initial value decomposition:
          \begin{itemize}
        \item At each iteration $k\geq 1$, we update the initial value $\phi_{k}\in L^{2}$ using the new smoother term $w_{k-1}(kT)$ and $v_{k-1}(kT)$ obtained from the previous iteration at the initial time $t=kT$ and merge $\phi_{k}$ with $\psi_{k}=e^{ikT\Delta}\psi_{0}$ to form the new approximation. See \eqref{newidk}.
        \item Solving INLS \eqref{INLS} with  $\phi_{k}\in L^{2}$ yields a solution $v_{k}$ as described in \eqref{gwpl2}-\eqref{v0qr2}. Further, solving modified  INLS \eqref{ivpMod}  with initial value $\psi_{k}$ and merging $v_{k},$ we get solution to INLS \eqref{INLS} in the interval $[kT,(k+1)T].$ See \eqref{ss}.
        \item We repeat the iteration by redefining $\phi_{k+1}$ and $\psi_{k+1},$ and maintaining control of the nonlinear interaction term $w_{k}((k+1)T)$ using Corollary \ref{winfty2} so that it can be absorbed into $v_{k}((k+1)T)$ without loosing  mass conservation of $\phi_{k+1}.$
        \end{itemize} 
      \item[-] Finally, we shall observe that $kT(N)$ can be made large enough under certain restriction on $p$. See Remark \ref{Whypso}.  This together with the blow-up alternative (Theorem \ref{lwp}), yields a global solution. 
     \end{itemize}
\end{Remark}

\begin{proof}[\textbf{Proof of Theorem \ref{gwp}}] If possible, we assume that,  for $u_0 \in M^{p, p'}$,  the solution established in 
Theorem \ref{lwp}  is not global in time. So we have the maximal time $T^*< \infty.$  In this case,  we shall  produce a solution $u$  of INLS \eqref{INLS} (to be defined in \eqref{ss} below), which will exist on a larger interval $[0, T_1]$ for $T_1> T^*.$ This will lead to a contradiction to the maximal time interval $[0, T^*).$\\

Denote the constant from Proposition \ref{w0exist} by $C=C(n,\alpha,b)$ and put
\begin{equation}\label{TN}
        T=T(N)=(3CN^{\beta})^{-\frac{\alpha}{\frac{4-2b-n\alpha}{4}-\frac{n(4-2b-n\alpha)}{4(\alpha+2)(n+2-b)}} }.
    \end{equation}
By splitting the initial data (see \eqref{dp}) and using  Proposition \ref{w0exist} for $\phi=\phi_{0}$ and $\psi=\psi_{0}$,  the solution of INLS \eqref{INLS} exists in the interval $[0, T(N)]$ and it is of the form
\[u=v_0 + e^{it\Delta} \psi_0 + w_0,\]
see \eqref{solnlocal}. We wish to extend our solution to the interval $[T(N), 2T(N)]$ by similar procedure but with the new initial data as the sum of the following two functions:    
\[ \phi_{1}=v_{0}(T)+w_{0}(T) \quad \text{and} \quad \psi_{1} =e^{iT\Delta}\psi_{0}.\]
More generally, we 
 wish to  extend our solution further by considering the following iterative procedure :
\begin{itemize}
    \item[--]We define $\phi_{k}$ and $\psi_k$ for $k \geq 1$ (for $k=0,\; \phi_{0}$ and $\psi_0$ are defined in \eqref{dp}) as follows:
    \begin{equation}\label{newidk}
 \phi_{k}= v_{k-1}(kT) + w_{k-1}(kT) \quad \text{and} \quad \psi_{k}= e^{ikT\Delta}\psi_{0}, 
\end{equation}
where
\begin{eqnarray}
    \begin{aligned}
     w_{k-1}(kT)=i \mu \displaystyle \int_{0}^{kT} e^{i(kT-\tau)\Delta} |x|^{-b}G(v_{k-1}+e^{ik\tau \Delta}\psi_{0},w_{k-1})(\tau) d\tau \nonumber \\
   + i \mu \displaystyle \int_{0}^{kT} e^{i(kT-\tau)\Delta} |x|^{-b}G(v_{k-1},e^{ik\tau \Delta}\psi_{0})(\tau) d\tau.
  \label{wKT}
    \end{aligned}
       \end{eqnarray}
       \item[--] Assume for $kT\leq T^*\footnote{Note that $(K-1)T\leq T^*.$ Since $T\leq 1,$ we will have $KT\leq T^{*}+1.$}, \phi=\phi_{k}$ and $\psi=e^{ikT\Delta}\psi_{0}$, $T$ satisfy all three conditions \eqref{c1}, \eqref{c2} and \eqref{c3} of Proposition \ref{w0exist}, where $k \in \{0, 1,\cdots, K-1\}.$
       \item[--]Let $v_{k}$ be INLS evolution of  $\phi_k$, and by construction 
\begin{equation}\label{ss}
   u(\cdot, t)= v_{k} (\cdot, t-kT ) + w_{k}(\cdot, t-kT) + e^{it \Delta} \psi_0,  \quad \text{if} \ t\in [kT, (k+1)T] 
\end{equation}
 defines a solution of INLS \eqref{INLS} for $k \in \{0, 1,\cdots, K-1\}$.  
 \end{itemize}
 We shall show that  the iterative process ends with  $KT >T^*.$ Since  $v_{K}$ and $e^{it\Delta} \psi_0$ are globally defined in appropriate spaces, we are left to handle nonlinear interaction term $w_{K}$ to extend the solution at $K$th iteration. To do this, we shall use Proposition \ref{w0exist} with $\phi=\phi_{K}$ and $\psi=e^{iKT\Delta}\psi_{0}.$  

In view of Remark \ref{nexp}, $T=T(N)\to 0$ as $N\to \infty$ (see \eqref{TN}), and so 
the  smallness condition \eqref{c1} is satisfied independently of $k$ for large $N.$\\
 Using Proposition \ref{mpq} and  \eqref{asi}, we have
\begin{align}\label{ref1} 
\|e^{it\Delta}\psi_{0}\|_{L^{\infty}([0,T^{*}+1],M^{\alpha+2,(\alpha+2)'}) } &\lesssim_{n,T^*}\|\psi_{0}\|_{M^{\alpha+2,(\alpha+2)'}} \lesssim_{n} \frac{1}{N} \xrightarrow{N \to \infty} 0.
    \end{align}
Inserting $e^{ikT\Delta}\psi_{0}$ in the right hand side of \eqref{c3}, we have
\begin{eqnarray*}
\begin{aligned}
    \left( \|e^{ikT\Delta}\psi_{0}\|_{M^{\alpha+2,(\alpha+2)'}} \right)^{-\frac{\alpha}{\frac{4-2b-n\alpha}{4} - \frac{n(4-2b-n\alpha)}{4(\alpha+2)(n+2-b)} + \frac{n\alpha^{2}}{4(\alpha+2)}}} \gtrsim_{n} N^{{\frac{\alpha}{\frac{4-2b-n\alpha}{4} - \frac{n(4-2b-n\alpha)}{4(\alpha+2)(n+2-b)} + \frac{n\alpha^{2}}{4(\alpha+2)}}}} \xrightarrow{N \to \infty} \infty.
   \end{aligned}
\end{eqnarray*}  
Since the lower bound is independent of $k$ and $T \xrightarrow{N \to \infty} 0,$  condition \eqref{c3} holds for sufficiently large $N.$ \\

\noindent 
Thus, we either have $KT> T^*$ or condition \eqref{c2} fails in the last iterative step $k=K,$ i.e.
\begin{equation}\label{aim11}
    3CN^{\beta}< \|\phi_{K}\|_{L^2}+ \|e^{iKT\Delta}\psi_{0}\|_{M^{\alpha+2,(\alpha+2)'}}.
\end{equation} 
Considering \eqref{ref1}, \eqref{aim11} can be written as
\begin{equation}\label{aim}
    3CN^{\beta}< \|\phi_{K}\|_{L^2}+CN^{\beta}.
\end{equation}
We claim that even under condition \eqref{aim},   we will have $KT>T^*$. This  clearly lead to a contradiction to the definition of $T^*.$ \\

\noindent
In view of the construction of $\phi_k$ and Corollary \ref{winfty2}, we note that $\phi_k \in L^2$ for $k\in \{0, 1,\cdots, K-1\}.$ Now exploiting the conservation \eqref{mass} and Corollary \ref{winfty2} (for  $w=w_{k}$ and $\psi=e^{ikT\Delta}\psi_{0}, 0\leq k \leq K-1$), we have
    \begin{align}
        \|\phi_{K}\|_{L^{2}}  &\leq \|v_{K-1}\|_{L^{\infty}_{[(K-1)T,KT]}L^{2}}+\|w_{K-1}\|_{L^{\infty}_{[(K-1)T,KT]}L^{2}} \nonumber\\
        &= \|\phi_{K-1}\|_{L^2} + \|w_{K-1}\|_{L^{\infty}_{[(K-1)T,KT]}L^2}\nonumber\\
        & \leq \|v_{K-2}\|_{L^{\infty}_{[(K-2)T,(K-1)T]}L^{2}}+\|w_{K-2}\|_{L^{\infty}_{[(K-2)T,(K-1)T]}L^{2}} +\|w_{K-1}\|_{L^{\infty}_{[(K-1)T,KT]}L^{2}} \nonumber\\
        &= \|\phi_{K-2}\|_{L^2}+\|w_{K-2}\|_{L^{\infty}_{[(K-2)T,(K-1)T]}L^{2}} +\|w_{K-1}\|_{L^{\infty}_{[(K-1)T,KT]}L^{2}} \nonumber\\
        & \leq \cdots \leq  \|\phi_{0}\|_{L^{2}}+\sum_{k=0}^{K-1}\|w_{k}\|_{L^{\infty}_{[kT,(k+1)T]}L^{2}} \nonumber \\
        & \lesssim_{n, \alpha,b} CN^{\beta}+T^{\frac{n\alpha}{4(\alpha+2)}}\sum_{k=0}^{K-1}\|e^{ikT\Delta}\psi_{0}\|_{M^{\alpha+2,(\alpha+2)'}}\nonumber\\
        &\lesssim_{n,T^{*}} CN^{\beta}+T^{\frac{n\alpha}{4(\alpha+2)}}K\frac{C}{N}\label{secondest}.
    \end{align}
    In the last two inequalities, we have used \eqref{asi} and \eqref{ref1}.
   Thus, using \eqref{TN}, \eqref{aim} can be expressed as
 \begin{align}
     KT &\gtrsim_{n,\alpha,b,T^*} N^{1+\beta}T^{1-\frac{n\alpha}{4(\alpha+2)}} \approx N^{1+\beta \left(1-\frac{\alpha\left(1-\frac{n\alpha}{4(\alpha+2)}\right)}{\frac{4-2b-n\alpha}{4}-\frac{n(4-2b-n\alpha)}{4(\alpha+2)(n+2-b)}}\right)}\nonumber\\
    &\label{Npower}= N^{1-\beta \left(-1+\frac{\alpha\left(1-\frac{n\alpha}{4(\alpha+2)}\right)}{\frac{4-2b-n\alpha}{4}-\frac{n(4-2b-n\alpha)}{4(\alpha+2)(n+2-b)}}\right)}.
\end{align}
Note that $N$ can be taken arbitrarily large. For any $\beta$ satisfying 
\begin{align}\label{betarange}
0 < \beta <
\begin{cases}
 \eta \quad &\text{if}\quad  \alpha-\frac{n\alpha^2}{4(\alpha+2)}-\frac{4-2b-n\alpha}{4}+\frac{n(4-2b-n\alpha)}{4(\alpha+2)(n+2-b)}>0\\
\infty \quad &\text{otherwise} 
\end{cases}
\end{align}
where
$$\eta =\frac{\frac{4-2b-n\alpha}{4}-\frac{n(4-2b-n\alpha)}{4(\alpha+2)(n+2-b)}}{\alpha-\frac{n\alpha^2}{4(\alpha+2)}-\frac{4-2b-n\alpha}{4}+\frac{n(4-2b-n\alpha)}{4(\alpha+2)(n+2-b)}}~,$$
the exponent of $N$ is positive in \eqref{Npower}, we get $KT>T^*$. This concludes the proof of Theorem \ref{gwp}.  
\end{proof}
\begin{Remark}\label{Whypso}
\begin{enumerate}
    \item Recall $\beta$ defined in terms of $p$ in  \eqref{betap}.  The range of $\beta$ in \eqref{betarange} in turn decides the range of $p.$ Thus, we have $p \in (2,p_{max})$  and 
     \begin{align*}p_{\max}=
         \begin{cases}
            \frac{4\alpha+8-n\alpha}{2\alpha+2+b+\frac{n(4-2b-n\alpha)}{2(\alpha+2)(n+2-b)}}  &\text{if}\quad \beta = \eta \\
             \alpha+2 &\text{if}\quad \beta=\infty. \\
         \end{cases}
     \end{align*}
      This justifies the choice of \(p\)  in Theorem \ref{gwp}.
      \item It is easy to verify that $2<p_{max}=\frac{2}{1-\frac{4\zeta}{4\alpha+8-n\alpha}}$ when $\beta=\eta$ and $\zeta=\frac{4-2b-n\alpha}{4}-\frac{n(4-2b-n\alpha)}{4(\alpha+2)(n+2-b)}$, since $0< \frac{4\zeta}{4\alpha+8-n\alpha}<1$ by our assumption on $b$ and $\alpha.$ 
\end{enumerate}

\end{Remark}
We shall now prove Proposition \ref{w0exist} and then Corollary \ref{winfty2}. 
\begin{proof}[\textbf{Proof of Proposition \ref{w0exist}}] 
Recall $Y(T)$ defined in \eqref{YT}.
     Define $$B(A,T)=\{u\in Y(T):\|u\|_{Y(T)}\leq A\}$$
    such that $w \in B(A,T)$ with $A>0$ (to be choosen later), $T$ be the minimum of the right-hand sides of the conditions \eqref{c1},  \eqref{c2} and \eqref{c3}(w.l.o.g. we may assume).
    Further, define
    $$\Gamma(w):= i \mu \int_{0}^{t} e^{i(t-\tau)\Delta}|x|^{-b} G(v + e^{it\Delta}\psi, w) \, d\tau +  i \mu  \int_{0}^{t} e^{i(t-\tau)\Delta}|x|^{-b} G(v, e^{it\Delta}\psi) \, d\tau. $$
    Firstly, we need to show that $\Gamma(w)\in B(A,T)$.\\
   Using Theorem \ref{SE1}, Lemma \ref{eg} and Lemma \ref{lemlwp} under the assumption \eqref{c1}, for any $v_{1},v_{2} \in \C,$ we have 
    \begin{align}
    \left\| \int_{0}^{t} e^{i(t-\tau)\Delta}|x|^{-b}G(v_{1},v_{2})(\tau) \, d\tau \right\|_{Y(T)}
   \lesssim_{n,\alpha,b} & \||v_{1}|^{\alpha}v_{2} \|_{L^{\gamma'_{1}}_{T}L^{\rho'_{1}}(B^{c})}  
   +  \|~|x|^{-b}|v_{1}|^{\alpha}v_{2} \|_{L^{\gamma'_{2}}_{T}L^{\rho'_{2}}(B)}   \nonumber\\ 
   +& \||v_{2}|^{\alpha+1} \|_{L^{\gamma'_{1}}_{T}L^{\rho'_{1}}(B^{c})} 
   + \|~|x|^{-b}|v_{2}|^{\alpha+1} \|_{L^{\gamma'_{2}}_{T}L^{\rho'_{2}}(B)} 
  \end{align}
  \begin{align}
   \lesssim\label{before4.10} & \left(T^{\frac{4 - 2b - n\alpha}{4} - \frac{n(4 - 2b - n\alpha)}{4(\alpha + 2)(n + 2 - b)}} \right) 
     \left(\|v_{1}\|^{\alpha}_{Y(T)} \|v_{2}\|_{Y(T)}+
     \|v_{2}\|^{\alpha + 1}_{Y(T)}\right)  \\
   \label{4.10}\lesssim_{\alpha} & \left( T^{\frac{4 - 2b - n\alpha}{4} - \frac{n(4 - 2b - n\alpha)}{4(\alpha + 2)(n + 2 - b)} + \frac{n\alpha}{4(\alpha+2)}} \right) 
     \|v_{1}\|^{\alpha}_{Y(T)} \|v_{2}\|_{L^{\infty}_{T}L^{\alpha+2}} \nonumber \\
   & + \left( T^{\frac{4 - 2b - n\alpha}{4} - \frac{n(4 - 2b - n\alpha)}{4(\alpha + 2)(n + 2 - b)} + \frac{n\alpha(\alpha + 1)}{4(\alpha+2)}} \right) 
     \|v_{2}\|^{\alpha + 1}_{L^{\infty}_{T}L^{\alpha+2}}.
\end{align}
In the last inequality, we have used the embedding $L^{\infty}_{T} \hookrightarrow L^{\frac{4(\alpha+2)}{n\alpha}}_{T}$ i.e. $\|\cdot\|_{L^{\frac{4(\alpha+2)}{n\alpha}}_{T}} \leq T^{\frac{n\alpha}{4(\alpha+2)}} \|\cdot\|_{L^{\infty}_{T}}.$
\\ Using the estimate \eqref{4.10} for $v_{1}=v , v_{2}=e^{i\tau \Delta}\psi$ along with \eqref{v0qr2}, Lemma \ref{srp} \eqref{embedding} and Proposition \ref{mpq} under the assumption \eqref{c1}, we obtain\\
    \begin{align*} \hspace{-7cm}\left|\left| \displaystyle\int_{0}^{t} e^{i(t-\tau)\Delta} |x|^{-b}G(v,e^{i\tau \Delta}\psi)(\tau) \, d\tau \right|\right|_{Y(T)} 
    \end{align*}
    \begin{align*}
   \lesssim_{\alpha,n,b}& \left( T^{\frac{4-2b-n\alpha}{4} - \frac{n(4-2b-n\alpha)}{4(\alpha+2)(n+2-b)} + \frac{n\alpha}{4(\alpha+2)}} \right) \|v\|_{Y(T)}^{\alpha}  \| e^{i\tau \Delta}\psi \|_{L^{\infty}_{T}L^{\alpha+2}} \\
     &+ \left( T^{\frac{4-2b-n\alpha}{4} - \frac{n(4-2b-n\alpha)}{4(\alpha+2)(n+2-b)} + \frac{n\alpha(\alpha+1)}{4(\alpha+2)}} \right) \| e^{i\tau \Delta}\psi \|_{L^{\infty}_{T}, 
     L^{\alpha+2}}^{\alpha+1} \\
     \lesssim_{\alpha,n} & \left( T^{\frac{4-2b-n\alpha}{4} - \frac{n(4-2b-n\alpha)}{4(\alpha+2)(n+2-b)} + \frac{n\alpha}{4(\alpha+2)}} \right) \| \phi \|_{L^{2}}^{\alpha}  \| \psi\|_{M^{\alpha+2, (\alpha+2)'}} \\
      &+\left( T^{\frac{4-2b-n\alpha}{4} - \frac{n(4-2b-n\alpha)}{4(\alpha+2)(n+2-b)} + \frac{n\alpha(\alpha+1)}{4(\alpha+2)}} \right) \| \psi \|_{M^{\alpha+2, (\alpha+2)'}}^{\alpha+1} \\
      = &T^{\frac{n\alpha}{4(\alpha+2)}} \| \psi \|_{M^{\alpha+2, (\alpha+2)'}}\left( T^{\frac{4-2b-n\alpha}{4} - \frac{n(4-2b-n\alpha)}{4(\alpha+2)(n+2-b)}} \right. \| \phi \|_{L^{2}}^{\alpha}   \\
      &+\left. T^{\frac{4-2b-n\alpha}{4} - \frac{n(4-2b-n\alpha)}{4(\alpha+2)(n+2-b)} + \frac{n\alpha^{2}}{4(\alpha+2)}}\| \psi \|_{M^{\alpha+2, (\alpha+2)'}}^{\alpha}  \right) \\
     \lesssim_{\alpha,n,b} &   T^{\frac{n\alpha}{4(\alpha+2)}} \| \psi \|_{M^{\alpha+2, (\alpha+2)'}}.
\end{align*}
The last inequality follows due to our assumptions \eqref{c2} and \eqref{c3}.
This suggests the choice of
    \begin{equation}\label{AA}
    A= \frac{3}{C(n,\alpha,b)}T^{\frac{n\alpha}{4(\alpha+2)}}\|\psi\|_{M^{\alpha+2,(\alpha+2)'}}.
    \end{equation}
    where $C=C(n,\alpha,b)$ is the same constant as in \eqref{c2} and \eqref{c3}, choosen such that\begin{equation}\label{gammaw1}
      \left|\left| \displaystyle\int_{0}^{t} e^{i(t-\tau)\Delta} |x|^{-b}G(v,e^{i\tau \Delta}\psi)(\tau) \, d\tau \right|\right|_{Y(T)}\leq \frac{A}{3}
    \end{equation}
    holds.
     Using the estimate \eqref{before4.10} for $v_{1}=v + e^{i\tau \Delta}\psi$ and $ v_{2}=w$ along with \eqref{v0qr2}, Lemma \ref{srp} \eqref{embedding}  and Proposition \ref{mpq} under the assumption \eqref{c1}, we have
\begin{align*}
    \hspace{-7cm} \left\| \int_{0}^{t}  e^{i(t-\tau)\Delta} |x|^{-b} G(v +  e^{i\tau \Delta}\psi, w)(\tau) \, d\tau \right\|_{Y(T)} 
\end{align*}
\begin{align*}
    & \lesssim_{\alpha,n,b} T^{\frac{4-2b-n\alpha}{4} - \frac{n(4-2b-n\alpha)}{4(\alpha+2)(n+2-b)}}   \left( \|v
    + e^{i\tau \Delta}\psi \|_{Y(T)}^{\alpha} \|w\|_{Y(T)} + \|w\|_{Y(T)}^{\alpha+1} \right) \\
    & \lesssim_{\alpha,n}  \|w\|_{Y(T)}\left\{T^{\frac{4-2b-n\alpha}{4} - \frac{n(4-2b-n\alpha)}{4(\alpha+2)(n+2-b)}}   \left( \left( \|\phi\|_{L^{2}} + \|\psi\|_{M^{\alpha+2, (\alpha+2)'}} \right)^{\alpha} +\|w\|_{Y(T)}^{\alpha} \right)\right\} \\
    & \lesssim_{\alpha,n,b}\|w\|_{Y(T)} \left\{\frac{1}{3} +T^{\frac{4-2b-n\alpha}{4} - \frac{n(4-2b-n\alpha)}{4(\alpha+2)(n+2-b)}+\frac{n\alpha^{2}}{4(\alpha+2)}}\|\psi\|^{\alpha}_{M^{\alpha+2,(\alpha+2)'}}   \right\}
\end{align*}
In the last inequality, we have used \eqref{c2} in the first summand (choosing $C$ in \eqref{c2} small enough) and substitute the norm of $w$ in $Y(T)$ to the power of $\alpha$ by $A^{\alpha}$, ($A$ given in \eqref{AA}) in the second summand. Consider second summand of last inequality under the assumption \eqref{c3} (choosing $C$ small enough in \eqref{c3}) to get
        \begin{equation}\label{gamma2}
            \left\| \int_{0}^{t} e^{i(t-\tau)\Delta} |x|^{-b} G(v + e^{i\tau \Delta}\psi, w)(\tau)\, d\tau \right\|_{Y(T)} \leq \frac{2A}{3}.
        \end{equation}
        Combining \eqref{gammaw1} and \eqref{gamma2}, we can say that $\Gamma(w)$ belongs to $B(A, T).$ Contractivity of $\Gamma$ follows similarly. Thus, by the Banach fixed-point theorem, we get a unique fixed point $w$ to the integral equation \eqref{w01} on the time-interval $[0,T].$
\end{proof}  

\begin{proof}[\textbf{Proof of Corollary \ref{winfty2}}] The proof follows from the
Strichartz estimates (Theorem \ref{SE1}) by replacing the norm defined with respect to $Y(T)$ by  $L^{\infty}_{T}L^{2}.$ Using an admissible pair $ (\infty,2)$ on the left hand side 
and the same pairs on the right hand side of \eqref{before4.10} and \eqref{4.10} in the proof of Proposition \ref{w0exist}, we obtain
\begin{align*}
\|w\|_{L^{\infty}_{T}L^{2}} \leq & \left|\left|\int_{0}^{t} e^{i(t-\tau)\Delta}|x|^{-b} G(v + e^{it\Delta}\psi, w)(\tau) \, d\tau \right|\right|_{L^{\infty}_{T}L^{2}}\\
&+\left|\left|\int_{0}^{t} e^{i(t-\tau)\Delta} |x|^{-b} G(v, e^{i\tau \Delta}\psi)(\tau) \, d\tau \right|\right|_{L^{\infty}_{T}L^{2}}\\
&\lesssim_{\alpha,n,b} T^{\frac{n\alpha}{4(\alpha+2)}}\|\psi\|_{M^{\alpha+2,(\alpha+2)'}}.
\end{align*}This completes the proof of Corollary \ref{winfty2}.
\end{proof}

{\bf Acknowledgments:} The second author acknowledges the financial support from the University Grants Commission (UGC), India (file number 201610135365) for pursuing Ph.D. program.
The second and third authors would like to express their gratitude to the Bhaskaracharya Mathematics Laboratory and the Brahmagupta Mathematics Library within the Department of Mathematics at IIT Indore, which are supported by the DST FIST Project (file number SR/FST/MS-I/2018/26). 
The third author would like to thankfully acknowledge the financial support from the Matrics Project of DST (file number 2018/001166).
\bibliographystyle{plain}
\bibliography{ msnls.bib}

\end{document}